%% file: SmallKnots.tex
\documentclass[11pt]{amsart}
\usepackage[margin=1.25in]{geometry}                
\usepackage{graphicx}

\makeatletter
\newtheorem*{rep@theorem}{\rep@title}
\newcommand{\newreptheorem}[2]{%
\newenvironment{rep#1}[1]{%
 \def\rep@title{#2 \ref{##1}}%
 \begin{rep@theorem}}%
 {\end{rep@theorem}}}
\makeatother

\date{\today}

\include{Packages}

\include{Environs}

\include{Commands}

\usepackage{wrapfig}
\graphicspath{{figures_pvw/}}

\renewcommand{\int}{\mathrm{int}}
\newcommand{\N}{\mathcal{N}}

\renewcommand{\phi}{\varphi}

\renewcommand{\path}{C}

\newcommand{\B}{\mathcal{B}}
\newcommand{\Mf}{M_{\gamma\setminus\B}}

\newcommand{\Mb}{M_{\gamma\setminus \{B_q\}}}
\newcommand{\hatMb}{\widehat{M}_{\gamma\setminus \{B_q\}}}
\newcommand{\piSig}{\Sigma_\circ}

\newcommand{\I}{\mathfrak{i}}
\newcommand{\piSighat}{\widehat{\Sigma}_\circ}
\newcommand{\Sighat}{\widehat{\Sigma}}

\title{Small knots of large Heegaard genus}
\author[W. Worden]{William Worden}
\address{Department of Mathematics\\
Rice University MS-136\\
1600 Main St.\\
Houston, TX 77005}
\email{\href{mailto:william.worden@rice.edu}{william.worden@rice.edu}}

\begin{document}

\begin{abstract}
Building off ideas developed by Agol, we construct a family of hyperbolic knots $K_n$ whose complements contain no closed incompressible surfaces and have Heegaard genus exactly $n$. These are the first known examples of such knots. Using work of Futer and Purcell, we are able to bound the crossing number for each $K_n$ in terms of $n$.
	
\end{abstract}

\maketitle
\date\today

\input{intro}

\input{background}

\input{construction}

\input{braids}

\input{proof}

\bibliographystyle{alpha2}
\bibliography{biblio}

\end{document}

%% file: Packages.tex
\usepackage{bigints}
\usepackage{amssymb}
\usepackage{mathdots}
\usepackage{subcaption}
\usepackage{tikz}
\usepackage{extarrows}
\usepackage[all]{xy}
\usepackage[justification=raggedright]{caption}
\usepackage{verbatim}
\usepackage{mathtools}
\usepackage{bbm}
\usepackage{amsthm}
\usepackage{nccmath}
\usepackage[color=yellow]{todonotes}
\usepackage{pgf}
\usepackage[pscoord]{eso-pic}
\usepackage[colorlinks,citecolor=blue,linkcolor=blue]{hyperref}

\usepackage[nameinlink]{cleveref}
\usepackage{comment}
\usepackage{marginnote}
\usepackage{import}
\usepackage{listings}
\usepackage{color}


%% file: Environs.tex
\theoremstyle{definition}
\newtheorem{thm}{Theorem}[section]

\newtheorem{lem}[thm]{Lemma}

\newtheorem{cor}[thm]{Corollary}

\newreptheorem{thm}{Theorem}

\theoremstyle{remark}
\newtheorem*{rem}{Remark}
\newtheorem*{clm}{Claim}

%% file: Commands.tex
\renewcommand{\SS}{\mathbb{S}}

\newcommand{\RR}{\mathbb{R}}

\newcommand{\CC}{\mathbb{C}}

\newcommand{\ZZ}{\mathbb{Z}}

\newcommand*\rfrac[2]{{}^{#1}\!/_{#2}}     
\newcommand{\del}{\partial}

\newcommand{\qt}[2]{\small\text{\raisebox{.5ex}{$#1/_{ #2}$\normalsize}}}   

\newcommand{\C}{\mathcal{C}}

\newcommand{\PL}{\mathcal{PL}}

\newcommand{\define}{\textbf}

%% file: intro.tex
\section{Introduction}\label{sec:intro}

Closed incompressible surfaces in irreducible 3-manifolds have been a subject of great interest since the notion was introduced by Haken in \cite{Hak68}. Most notably, Thurston  in \cite{Thu82a} proved his celebrated geometrization conjecture for the case of closed irreducible 3-manifolds containing an incompressible surface (i.e., Haken manifolds) 20 years before the theorem was proved in full generality by Perelman \cite{Per02,Per03a,Per03b}. This piece of history suggests a need to better understand irreducible 3-manifolds containing \emph{no} closed incompressible surfaces (i.e., small manifolds). Answering a question of Reid, Agol in \cite{Ago03} constructed the first examples of small link complements having arbitrarily many components. By Dehn filling such a link, Agol was able to give the first examples of small closed manifolds having large Heegaard genus. Previous to Agol's result, many examples of small manifolds had been constructed (c.f. \cite{FH82}, \cite{CJR82}, \cite{Oer84}, \cite{Dun99}, \cite{HM93}, \cite{Lop92}, \cite{Lop93}), but the largest known Heegaard genus of a small manifold was 3 (in particular, by \cite{CS99}, this is true of many of the punctured torus bundles of \cite{FH82,CJR82}).

The small links Agol constructs are $n$-braids having $n$ components. Such braids are easily seen to have Heegaard genus at least $\frac{n}{2}$, and work of Rieck \cite{Rie00} (which generalized work of Moriah and Rubinstein \cite{MR97}) implies that one can find filling slopes such that the Heegaard genus does not decrease upon Dehn filling. By work of Hatcher \cite{Hat82}, the filling slopes can also be chosen so that the filled manifold remains small. While Agol's methods can easily be extended to construct small knots which are $n$-braids, one immediately loses the ability to control Heegaard genus. Nonetheless, Agol in \cite{Ago03} suggests that it should be possible to generalize his construction to exhibit examples of small knots having large Heegaard genus.

Although demonstrating the existence of small knots with large Heegaard genus is our main goal, we are actually able to prove the following stronger result:

\begin{thm}\label{thm:main}
For every $l\ge 1$, there exists a family $\{K_N\}_{N=4}^\infty$ of small $l$-component links such that $\SS^3\setminus K_N$ has Heegaard genus $n:=Nl$. Furthermore, $K_N$ is an $n$-bridge link, and can be constructed to have fewer than $4\pi n^5$ crossings. 
\end{thm}

Thus for any $l$, we can build small $l$-component links, with Heegaard genus as large as desired. We immediately get the following corollary:

\begin{cor}\label{cor:main}
There exists a family $\{K_n\}_{n=4}^\infty$ of small knots such that $\SS^3\setminus K_n$ has Heegaard genus $n$. Furthermore, $K_n$ is an $n$-bridge knot, and can be constructed to have fewer than $4\pi n^5$ crossings. 
\end{cor}

Proving \Cref{thm:main} relies fundamentally on Agol's work in \cite{Ago03}. The main novelty of our work lies in our approach to controlling Heegaard genus, for which our main tool is a theorem of Rieck and Sedgwick \cite{RS01}. One advantage of our approach is that it allows us to determine the Heegaard genus on the nose, in contrast to Agol's bound of $\frac{n}{2}$. By Dehn filling each $\SS^3\setminus K_N$ along an appropriate slope system, one can obtain small closed 3-manifolds having Heegaard genus exactly $n$.

The bound on crossing number given in \Cref{thm:main} is calculated using recent work of Futer and Purcell, which effectivizes the main results of \cite{MR97} and \cite{RS01}.

\begin{rem}
There is an alternate approach to proving \Cref{thm:main}, suggested to us by Dave Futer. After using Agol's construction to obtain small $l$-component links, one could then appeal to work of either Biringer-Souto \cite{BS16} or Bachman-Schleimer \cite{BS05}. In particular, the link $K_N$ is built by Dehn filling a braid augmented by loops, each of which encircles some of its strands (see \Cref{fig:M_C}). Dehn filling these loops gives a fibered manifold with fiber a punctured disk, which is an $l$-component braid in a solid torus (an $(l+1)$-component link). Long Dehn fillings can be shown to produce a monodromy for this fibered manifold which has large translation distance in the curve complex, and either of \cite{BS16} or \cite{BS05} then implies that the Heegaard genus is $n+1$ (the genus is subsequently reduced to $n$ when the extra component is filled). This approach comes with a caveat, though. Both results \cite{BS16} and \cite{BS05} apply, as written, only to closed hyperbolic 3-manifolds. Thus one would first have to generalize one of these results to cusped hyperbolic 3-manifolds. As our proof is more direct, and more elementary, we feel it is the right approach. 
	
\end{rem}

The paper is organized as follows. In \Cref{sec:bg}, we cover some needed background on pants decompositions, essential surfaces, and Heegaard splittings. In \Cref{sec:const} we describe Agol's construction and state a version of the main lemma in Agol's paper. In \Cref{sec:braids} we describe our particular application of Agol's construction to obtain a family of links which we will Dehn fill to get the desired links $K_N$. Finally, in \Cref{sec:proof} we prove \Cref{thm:main}. For the sake of clarity, we push most of the proof of \Cref{thm:main} into a series of lemmata, before finishing up the proof at the end of the section.

\subsection{Acknowledgements}

We thank Alan Reid for bringing to our attention Agol's work, and for many helpful conversations throughout the course of this project. We also thank Dave Futer for showing us the alternate approach outlined in the preceding remark, and for other helpful conversations. Finally, we thank the referees for many helpful suggestions that improved the exposition.

%% file: background.tex
\section{Preliminaries}\label{sec:bg}

Throughout the paper, all surfaces and 3-manifolds will be assumed to be orientable, and all isotopies will assumed to be smooth.

\subsection{Pants Decompositions}

Given a connected compact surface $\Sigma$, we say $\Sigma$ is of type $(g,n)$ if it is a genus $g$ surface with $n$ boundary components (and no punctures). Let $\Sigma$ be such a surface, with $\chi(\Sigma)<0$. We define a \define{pants decomposition} $P$ of $\Sigma$ to be a choice of disjoint, simple closed curves $\alpha_1,\dots, \alpha_m$ on $\Sigma$ that cut $\Sigma$ into a disjoint union of surfaces of type $(0,3)$, called \define{pants}. Our definitions in this section follow those of \cite{HLS00}, though we note that pants decompositions appear in the literature much earlier \cite{HT80}. Any pants decomposition of a surface of type $(g,n)$ consists of $3g-3+n$ curves, and has $|\chi(\Sigma)|$ pants. Let $\alpha$ be a curve in a pants decomposition $P$ for $\Sigma$. If we cut $\Sigma$ along all curves \emph{except} $\alpha$, then the connected component $\Sigma_0$ containing $\alpha$ will either be a surface of type $(1,1)$ (that is, a one-holed torus), or a surface of type $(0,4)$ (a four-holed sphere). In the first case, let $\alpha'\subset \Sigma_0$ be a simple closed curve that transversely intersects $\alpha$ once. If we replace $\alpha$ by $\alpha'$, we get a new pants decomposition $P'$. In this case we say that $P'$ is obtained from $P$ via a \define{simple move}, or $S$-move (see \Cref{fig:S_move}). If, on the other hand, $\Sigma_0$ is a four-holed sphere, and $\alpha'$ is a curve that transversely intersects $\alpha$ twice (and cannot be made disjoint from $\alpha$ by isotopy), then replacing $\alpha$ with $\alpha'$ again gives a new pants decomposition $P'$, which we say is obtained by an \define{associative move}, or $A$-move, on $P$ (see \Cref{fig:A_move}).

\begin{figure}[h]
	\begin{subfigure}{.49\textwidth}
 		\centering
   		\includegraphics[width=.8\textwidth]{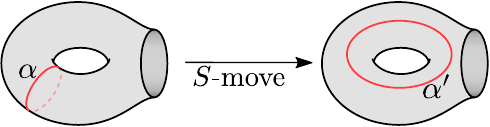}
   		\caption{}
   		\label{fig:S_move}
	\end{subfigure}
	\begin{subfigure}{.49\textwidth}
 		\centering
   		\includegraphics[width=.8\textwidth]{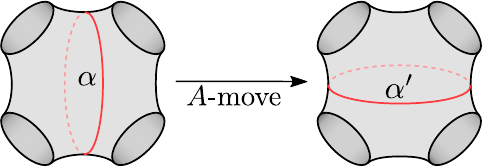}
   		\caption{}
   		\label{fig:A_move}
	\end{subfigure}	
	\caption{Left: An $S$-move. Right: An $A$-move.}
	\label{fig:moves}
\end{figure}

\subsection{Essential surfaces}\label{sec:essential}

Let $M$ be a compact irreducible manifold, and $\Sigma$ a surface properly embedded in $M$. A \define{compression disk} for $\Sigma$ is an embedded disk $D$ in $M$ with boundary $\del D=D\cap\Sigma$, such that $\del D$ does not bound a disk in $\Sigma$. If $\Sigma$ has no compression disks, then it is \define{incompressible}. Otherwise it is \define{compressible}. A disk $D$ with $\mathrm{int}(D)\cap\Sigma=\emptyset$ is a \define{boundary compression disk} if $\del D$ is a union of two simple arcs $\alpha$ and $\beta$ with $\alpha\subset \Sigma$ and $\beta \subset \del M$, such that $\alpha$ is not isotopic into $\del \Sigma\subset \Sigma$. If $\Sigma$ has no boundary compression disks, then it is \define{boundary incompressible}. If $\Sigma$ is incompressible and boundary incompressible, and is not isotopic into $\del M$, then it is \define{essential}.

\begin{thm}\cite{Hat82}\label{thm:Hatcher}
Let $M$ be an orientable, compact, irreducible 3-manifold with $\del M$ a union of $n$ tori. Then the projective classes of curve systems in $\del M$ which bound incompressible, $\del$-incompressible surfaces in $M$ form a dense subset of a finite (projective) polyhedron in $\PL(\del M)\cong \SS^{2n-1}$ of dimension $\le n-1$.	
\end{thm}

The above result of Hatcher relies fundamentally on a result of Floyd and Oertel \cite{FO84}, which shows that any such surface is carried by finitely many branched surfaces in $M$. The slopes that must be avoided in \Cref{thm:Hatcher} are precisely those that are carried by the boundaries of the branched surfaces given by Floyd--Oertel. 

We are interested in the case in which $M$ has $m+1$ boundary components, $m$ of which are filled along slopes of the form $\frac{1}{s}$, $0\ne s\in \ZZ$. Since we will only be concerned with surfaces that do not meet the unfilled boundary component of $M$, the subset that we need to avoid has dimension $\le m-1$. The space of such fillings, on the other hand, consists of integer lattice points in an affine subspace of $H_1(\del M; \RR)\cong \RR^{2m+2}$ of dimension $m$, which in $\PL(\del M)$ accumulates on a subspace of dimension $m-1$. Thus infinitely many of these projected lattice points are outside \emph{any} subspace of dimension $m-1$, and we can choose our filling slope system to avoid the polyhedron given by Hatcher's theorem. In fact, more is true: if $\gamma=(\frac{1}{s_1},\dots, \frac{1}{s_m},\infty)$ is a curve system in Hatcher polyhedron, then for some $i$, increasing $s_i$ by 1 will give a curve system outside the polyhedron. This is because there are $m$ such linearly independent $s_i$ to choose from, and the Hatcher polyhedron has dimension $m-1$.

\subsection{Heegaard splittings}	\label{sec:heegaard}

The main tool that will allow us to extend Agol's result to knots in $\SS^3$ is a result of Rieck and Sedgwick, which describes how Heegaard surfaces that appear after Dehn filling lead to essential surfaces in the unfilled manifold. To state this theorem we will need some terminology.

Let $M$ be a compact connected manifold with boundary a union of tori. A \define{Heegaard surface} $H\subset M$ for $M$ is an embedded closed surface that cuts $M$ into two \define{compression bodies}, thus giving a \define{Heegaard splitting} for $M$. For our purposes, a compression body is defined to be a connected, compact 3-manifold obtained by attaching 1-handles to the $T^2\times \{1\}$ boundaries of a disjoint union of thickened tori $T^2\times I$, or to a ball (in which case we get a \define{handlebody}). The \define{Heegaard genus} of $M$, which we will denote by $g(M)$, is the minimal genus over all Heegaard surfaces for $M$.

Now let $T$ be a boundary torus of $M$, and let $\alpha$ be a Dehn filling slope on $T$. Denote by $M_\alpha$ the manifold resulting from Dehn filling along $\alpha$. Let $H$ be a Heegaard surface for $M_\alpha$, and let $\rho$ be the core of the filling solid torus. Then one of the following holds:

\begin{itemize}

\item[(1)] 	$\rho$ is isotopic into $H$ and $H$ is a Heegaard surface for $M$.

\item[(2)] $\rho$ is isotopic into $H$ and $H$ is \emph{not} a Heegaard surface for $M$, or

\item[(3)] $\rho$ is \emph{not} isotopic into $H$ and $H$ is \emph{not} a Heegaard surface for $M$.

\end{itemize}

Our goal is to construct knots having large Heegaard genus, by Dehn filling $m$ components of an $m+1$-component link complement, which has Heegaard genus $\ge \frac{m}{2}$. Therefore we would like to avoid (3) and, when possible, (2). By \cite[Corollary 4.2]{Rie00}, only finitely many filling slopes need to be excluded to ensure that (3) is avoided, so if we are filling multiple boundary components, we only need to avoid finitely many slopes on each. The main obstacle then is (2). When (2) happens, there is a unique (isotopy class of) curve $\beta$ on $T\subset \del M$ such that $\beta$ is isotopic onto $H$, regarded as a surface in $M$. If we cut $H$ along $\beta$ we get a surface $H^\ast$ with two boundary components on $T$ of slope $\beta$. In this case we will have $g(M_\alpha)=g(M)-1$, and this will be true for \emph{every} filling slope on $T$ that intersect $\beta$ once. Such a line of Dehn surgery slopes is called a \define{destabilization line} for $T$. Since we will always be filling along a slope of the form $\frac{1}{s}$, our main concern is when the longitude of a boundary torus defines a destabilization line, since in this case every slope of this form results in a reduction in Heegaard genus. 

The surface $H^\ast$ described above is called an \define{almost Heegaard} surface. It requires only a single stabilization, which increases genus by 1, to become a Heegaard surface for $M$. When a filling is of type (3), it is called a \define{bad} filling. For bad fillings, one has no control over how much the genus may drop. Fillings of type (1) and (2) are called \define{good} fillings, and are further characterized by the following theorem of Rieck and Sedgwick:

\begin{thm}\cite[Lemma 3.4, Theorem 5.1]{RS01}\label{thm:RS}
Let $M$ be a compact, orientable, acylindrical manifold with boundary a union of tori, and let $(M_\alpha,H)$ be a good filling. Then one of the following holds:
\begin{itemize}
\item[(1)] $H$ is a Heegaard surface for $M$ (perhaps after an isotopy in $M_\alpha$), and $H$ is boundary compressible.
\item[(2)] the slope of the almost Heegaard surface $H^\ast$ is the boundary slope of a separating essential surface of genus less than or equal to that of $H^\ast$.	
\end{itemize}

The essential surface given by (2) will have exactly two boundary components, though it may be disconnected. In fact, such a surface is obtained by compressing $H^\ast$, then throwing away components that do not have boundary, as is immediately clear from the proof given in \cite{RS01}.

\end{thm}

%% file: construction.tex
\section{Agol's construction}\label{sec:const}
Let $F$ be a surface of type $(g,n)$, and let $\phi:F\to F$ be a homeomorphism. Denote by $T_\phi$ the mapping torus $\qt{(F\times I)}{\{(x,0)\sim (\phi(x),1)\}}$. Let $P$ be a pants decomposition of $F$, and let $\phi(P)$ be the image of $P$ under the homeomorphism $\phi$. Let $C=\{P=P_0,P_1,\dots,P_m=\phi(P)\}$ be a path in the pants graph from $P$ to $\phi(P)$, such that no (isotopy class of) curve $\alpha$ on $F$ is contained in every $P_i$. Let $\beta_i$ be the curve of $P_i$ that is replaced when passing from $P_i$ to $P_{i+1}$, $0\le i\le m-1$. The requirement that no curve $\alpha$ is contained in every $P_i$ is equivalent to requiring that every curve in $P$ appears as some $\beta_i$. For each curve $\beta_i$, let $L_i=\beta_i\times \{\rfrac{i}{m}\}$, and remove a neighborhood $\N(L_i)$ of $L_i$ from $T_\phi$ to obtain a compact manifold $M_C=T_\phi\setminus \bigcup_i \N(L_i)$. Define the \define{horizontal} boundary of $M_C$ to be those boundary components coming from drilling out neighborhoods $\N(L_i)$, and define the \define{vertical} boundary to be the boundary components coming from $T_\phi$.

Here is another way to build $M_C$. First, let $T_1$ be a torus with one disk removed, and let $B_S$ be obtained by thickening $T_1$ to get $T_1\times I$, then retracting the annulus $\del T_1\times I$ to its core curve. Similarly, for $S_4$ a sphere with four disks removed, let $B_A$ be the result of thickening $S_4$ then retracting the annuli $\del S_4 \times I$ to their core curves. Begin with the surface $F$, marked by the curves of the pants decomposition $P=P_0$. Suppose $P_0 \to P_1$ is an $S$-move, so that the curve $\beta_0$ lies on a subsurface of $F$ which is a one-holed torus with boundary some other curve $\beta'$ of $P_0$. Now glue a copy of $B_S$ to this subsurface along $T_1\times \{0\}$, so that the pinched $\del T_1\times I$ glues to $\beta'$, and mark $T_1\times \{1\}$ with the curve that replaces $\beta_0$. If $P_0 \to P_1$ is an $A$-move, the operation is similar, except the subsurface is a four-punctured sphere. Continue in this way for each move $P_i\to P_{i+1}$. Since every curve in $P_0$ is eventually replaced, the result is a copy of $F\times I$ with embedded curves $\{\beta_i\}_{i=0}^m$, which can be isotoped so that $\beta_i$ is on the fiber $F\times \{\frac{i}{m}\}$. The curves on the bottom glue to the curves on the top via $\phi$, and we obtained $M_C$ by drilling out a neighborhood of each curve. This shows that $M_C$ can be decomposed into \define{$A$-blocks} and \define{$S$-blocks} by cutting along pants. An $A$-block is defined here to be the complex obtained by removing from $B_A$ neighborhoods of two simple closed curves, one on $S_4\times \{0\}$ and one on $S_4\times\{1\}$, which intersect twice as curves on $S^4$, and a neighborhood of the pinched $\del S_4\times I$ (see \Cref{fig:decomp}, bottom left). An $S$-block is defined similarly. In general, we will call both $A$-blocks and $S$-blocks \define{pants blocks}.

\begin{figure}[h]
 	\centering
   	\includegraphics[width=.95\textwidth]{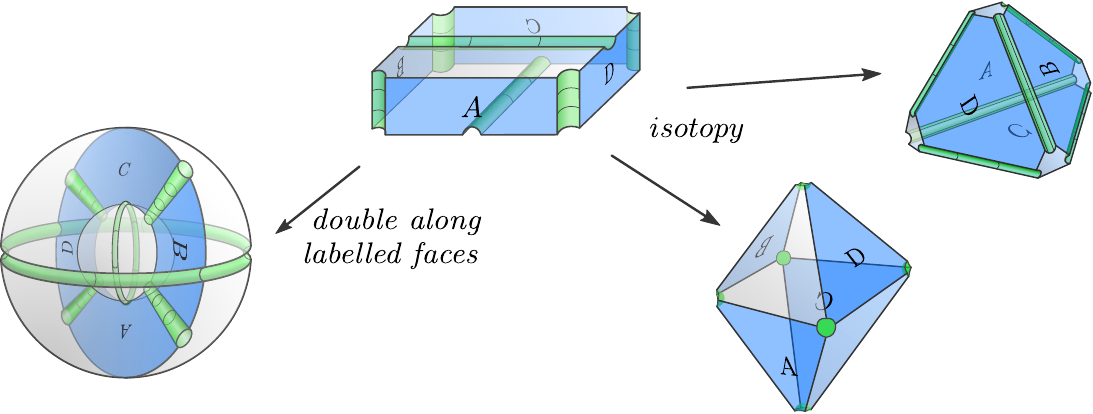}
   	\caption{An $A$-block (lower left) is obtained by doubling a vertex-truncated octahedron (lower right) along some of its faces. Alternatively, an ideal $A$-block is obtained by doubling an ideal octahedron along some of its faces. Since a truncated octahedron can be identified with a tetrahedron minus a neighborhood of its one-skeleton, an $A$-block can also be viewed as a doubling of such a tetrahedron.}
 \label{fig:decomp}
\end{figure}

If $R$ is an $A$-block or $S$-block of $M_C$, then we will call $R\cap \mathrm{int}(M_C)$ an \define{ideal} $A$- or $S$-block (or, in general, an ideal pants block). Each ideal pants block can be obtained by gluing either one or two ideal octahedra, as is shown in \Cref{fig:decomp}. If these octahedra are given the structure of regular ideal hyperbolic octahedra, then the boundary of each ideal pants block is totally geodesic, and gluing them along their geodesic faces gives $\mathrm{int}(M_C)$ a complete, finite volume, hyperbolic structure, as was observed by Agol \cite{Ago03}.

\subsection{Tubing pants}

In this section, $M_f$ will be a manifold obtained by Dehn filling any number of the horizontal boundary components of $M_C$ (possibly none). Boundary components of $M_C$ have a natural longitude and meridian, which induce a framing on the boundary components of $M_f$. We can also refer to horizontal and vertical boundary components of $M_f$, according to whether the corresponding boundary component of $M_C$ is horizontal or vertical. We will say a surface $\Sigma \subset M_f$ has \define{horizontal} boundary if every component of $\del \Sigma$ is isotopic to either the meridian of a vertical boundary component, or the longitude of a horizontal boundary component. 

An annulus $A\subset M_f$ with interior disjoint from $\Sigma$ is called a \define{compression annulus} for $\Sigma$ if it has one boundary component $\del A^+$ in $\mathrm{int}(\Sigma)$ and the other boundary component $\del A^-$ on $\del M_f$ is horizontal, and cannot be isotoped into $\Sigma$ via an isotopy through which $\del A^-\subset \del M_f$. If $A$ is a compression annulus for $\Sigma$ then we can \define{compress along $A$} by  attaching $(A\times I)$ to $\Sigma$, then removing $A\times \mathrm{int}(I)$.
When a (possibly disconnected) surface $(\Sigma,\del\Sigma)\subset(M_f,\del M_f)$ has two horizontal boundary components $\alpha_1$ and $\alpha_2$ on a boundary torus $T\subset \del M_f$, we can \define{tube} $\Sigma$ along $T$ by attaching an annulus along the $\alpha_i$ and pulling it away from $\del M_f$. This introduces a compression annulus, along which we can \define{untube} $\Sigma$ by compressing the annulus (see \Cref{fig:tubing}). Note that tubing involves a choice of annulus, but we will see later that for our applications both choices will result in the same surface up to isotopy.

\begin{figure}[h]
	\begin{subfigure}{.7\textwidth}
 		\centering
   		\includegraphics[width=.8\textwidth]{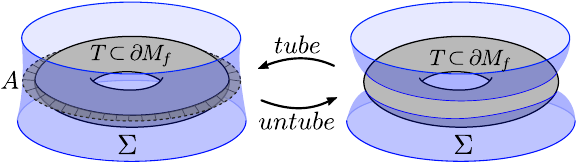}
   		\caption{}
   		\label{fig:tubing}
	\end{subfigure}
	\begin{subfigure}{.29\textwidth}
 		\centering
   		\includegraphics[width=.85\textwidth]{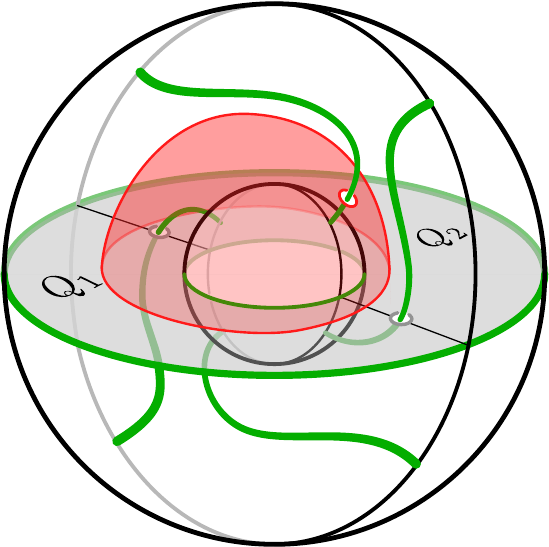}
   		\caption{}
   		\label{fig:quads}
	\end{subfigure}	
	\caption{Left: compressing $\Sigma$ along $A$ (i.e., untubing) by cutting along the dotted intersection produces two new boundary components. Conversely, tubing along a boundary torus produces a compression annulus. Right: two normal quads $Q_1$ and $Q_2$ glue together to give a twice punctured annulus $A_2$ in an $A$-block. In this case $A_2$ has a compression annulus, shown here as a red dome.}
	\label{fig:P1_to_P2}
\end{figure}

The following lemma shows that compressing along annuli preserves incompressibility for $M_f$ with hyperbolic interior.

\begin{lem}\label{lem:compress_annulus}
Let $M_f$ be a filling of $M_C$ with hyperbolic interior, and let $\Sigma$ be an incompressible surface. Then the surface $\Sigma'$ resulting from compressing $\Sigma$ along a compression annulus $A$ is also incompressible.	
\end{lem}

\begin{proof}
Let $\del A^+$ be the boundary component of $A$ lying on $\Sigma$. Suppose there is a (non-boundary) compressing disc $D$ for $\Sigma'$. We may assume that $\del D$ is bounded away from $\del\Sigma'$, so that $\del D\cap \del A^+=\emptyset$. Now suppose that $\del D$ bounds a disk $D_\Sigma$ in $\Sigma$. Since $\del A^+$ is disjoint from $\del D$, we must have that $\del A^+$ is contained in $\int(D_\Sigma)$ or $\del A^+\cap D_\Sigma=\emptyset$. In the first case we must have $\del D$ isotopic to $\del A^+$ and hence $\del D$ is isotopic through $A$ to an essential curve in $\del M_f$. Since $M_f$ has hyperbolic interior this is impossible \cite{Thu82a}. Therefore we are in the second case, so when we compress the annulus $A$ to get $\Sigma'$, $D_\Sigma$ is preserved as a disk in $\Sigma'$, contradicting our assumption that $D$ was a compression disk. Thus if $D$ is a compression disk in $\Sigma'$, it is also a compression disk in $\Sigma$, but this is impossible since $\Sigma$ was assumed to be incompressible.
\end{proof}

The following lemma is due to Agol \cite{Ago03}. Since our statement of this lemma is somewhat different than in \cite{Ago03}, and since it is of central importance to the proof of \Cref{thm:main}, we reproduce Agol's proof below.

\begin{lem}[\cite{Ago03}]\label{lem:tubing}
Let $M_f$ be a filling of $M_C$ with hyperbolic interior. Then every incompressible surface in $M_f$ that is isotopic into $M_C$ and has horizontal boundary is obtained by tubing pants in $M_C$.	
\end{lem}

\begin{proof}
Let $(\Sigma,\del \Sigma)\subset (M_f,\del M_f)$ be an incompressible surface in $M_f$ that is isotopic into $M_C$ and has horizontal boundary. First isotope $\Sigma$ into $M_C\subset M_f$. With this done, we will view $\Sigma$ from now on as a subset of $M_C$. We will untube $\Sigma$ along compression annuli in $M_C$, and show that the resulting surface is a disjoint union of pants. 

Since components of $\del \Sigma$ must be parallel on $\del M_C$ to pants boundaries, we can isotope $\Sigma$ away from the pants in a neighborhood of $\del \Sigma$. It follows that $\Sigma \cap \{\text{pants}\}$ is a disjoint union of essential simple closed curves. If $\alpha$ is such a curve, then $\alpha$ is parallel to a boundary component $\del Y$ of a pair of pants $Y$, since all closed curves on pants are boundary parallel. Thus $\alpha$ forms one boundary curve of an annulus $A\subset Y$ whose other boundary curve is on $\del Y$. Therefore we can untube $\Sigma$ along $A$, and the resulting surface will have one fewer curves of intersection with the pants (\Cref{fig:tubing} may be helpful). Doing this surgery at every such $\alpha$ results in a surface $\Sigma'$ that is disjoint from pants. Since $M_C$ is finite volume and hyperbolic, we may assume that no component of $\Sigma'$ is an annulus. By \Cref{lem:compress_annulus}, $\Sigma'$ is incompressible in $M_C$.

Since the pants partition $M_C$ into pants blocks, each component of $\Sigma'$ is contained in such a block. An ideal pants block is made up of ideal octahedra, so a pants block is made up of vertex-truncated octahedra. A truncated octahedron is the same as a tetrahedron with a neighborhood of its 1-skeleton removed, as shown in \Cref{fig:decomp}. With this point of view, we consider the shaded faces (blue in \Cref{fig:decomp}) to be the faces of the tetrahedra. Since pants blocks are obtained by gluing the shaded faces of the truncated octahedra, we can see them as manifolds triangulated by either one or two tetrahedra, with a neighborhood of the 1-skeleton of the triangulation removed. That is, a pants block is the exterior of the 1-skeleton of a triangulation of a closed manifold. Since $\Sigma'$ is incompressible in the pants block, it is normal (after isotopy) in this closed manifold by \cite[Claim 1.1]{Tho94}, and hence in the truncated tetrahedra it decomposes into (truncated) normal triangles and quads having normal arcs on shaded faces (see \cite{Gor01} or \cite{JR89} for background on normal surface theory). Since an $A$-block is obtained by doubling a truncated tetrahedron along its faces,  normal triangles glue to other triangles to give pants, which are isotopic to boundary pants. For the same reason, quads glue to other quads. The surfaces coming from pairs of quads are either, (1) isotopic to $S_4\times\{\rfrac{1}{2}\}$, and so are a tubing of two 3-punctured spheres isotopic into the boundary of the $A$-block, or (2) a twice punctured annulus $A_2$ with boundary the two curves on the $S_4$ boundaries of the $A$-block. In the latter case one can find a compression annulus, as shown in \Cref{fig:quads}. For $S$-blocks there is one quad that is isotopic to $T_1\times \{\rfrac{1}{2}\}$, and this is the only quad that can glue up coherently as part of a properly embedded surface. Triangles must glue up in pairs, and result in surfaces that have compression annuli, and untube to copies of $T_1\times \{\rfrac{1}{2}\}$. We leave further details for the $S$-block case to the reader, as these will not play a part going forward. This shows that $\Sigma'$ is a disjoint union of pants, and that $\Sigma$ was obtained by tubing together pants.
\end{proof}

%% file: braids.tex
\vspace{-5pt}
\section{Augmented braids}\label{sec:braids}

In this section we use Agol's construction to obtain explicit manifolds $M_C$, which will be used in \Cref{sec:proof} to build the links of \Cref{thm:main}. These manifolds are very similar to those used in Agol's proof, with the main difference being that we choose the monodromy $\phi$ so that the link resulting from filling the horizontal boundary has $l$ components. Although we describe the path in the pants graph differently than Agol, ours is essentially the same as his, but traced out (almost) twice.

Fix an integer $l\ge 1$, and let $n \in \{lN \mid N\in \ZZ_{\ge 1}, lN\ge 4\}$. Let $D_{z,r}$ be the disk of radius $r$ centered at $z\in\CC$, and let $D_n=D_{0,2}\setminus \left(\bigcup_{k=0}^{n-1} D_{e^{2\pi \I k/n},\frac{1}{n}}\right)$, where $\I=\sqrt{-1}$. That is, $D_n$ is a disk with $n$ small holes, equally spaced around a circle. Let $p_k$ be the boundary component of $D_n$ coming from removing the disk $D_{e^{2\pi i k/n},\frac{1}{n}}$, and let $q$ be the boundary component of $D_n$ coming from the boundary of $D_{0,2}$. Define a homeomorphism $\phi:D_n\to D_n$ by $z\mapsto e^{-2\pi \I l/n}z$, and let $T_\phi=\qt{(D_n\times I)}{(x,0)\sim (\phi(x),1)}$ be the associated mapping torus.

By our choice of $\phi$, it is clear that $T_\phi$ has exactly $l+1$ boundary components, since we chose $n$ to be a multiple of $l$. Let $L_q$ be the boundary component coming from $q$, which is fixed by $\phi$, and let $L_p=L_1\cup \dots \cup L_l$ be the union of boundary components coming from the $p_k$. We may think of $T_\phi$ as the exterior $\SS^3\setminus (L_q\cup L_p)$ of a closed braid in $\SS^3$, augmented with an additional unknotted component $L_q$. Alternatively, we can view $T_\phi$ as a closed braid in a solid torus, where $L_q$ is the boundary of the solid torus. We will generally take the latter view, and this is reflected in the figures. When we glue $D_n\times \{0\}$ to $\phi(D_n)\times \{1\}$, we can add a twist by any multiple of $2\pi$ without affecting the homeomorphism type of $T_\phi$. For the purpose of exposition, it will be convenient to include an extra twist by $4\pi$ in the counter-clockwise direction. With this convention, for $n=6$ and $l=1$, $T_\phi$ is the gray braid shown in \Cref{fig:M_C} (ignore for now the blue loops in the figure).

\begin{wrapfigure}[39]{r}{.33\textwidth}
	\vspace{5pt}
 	\raggedleft
   	\includegraphics[width=.29\textwidth]{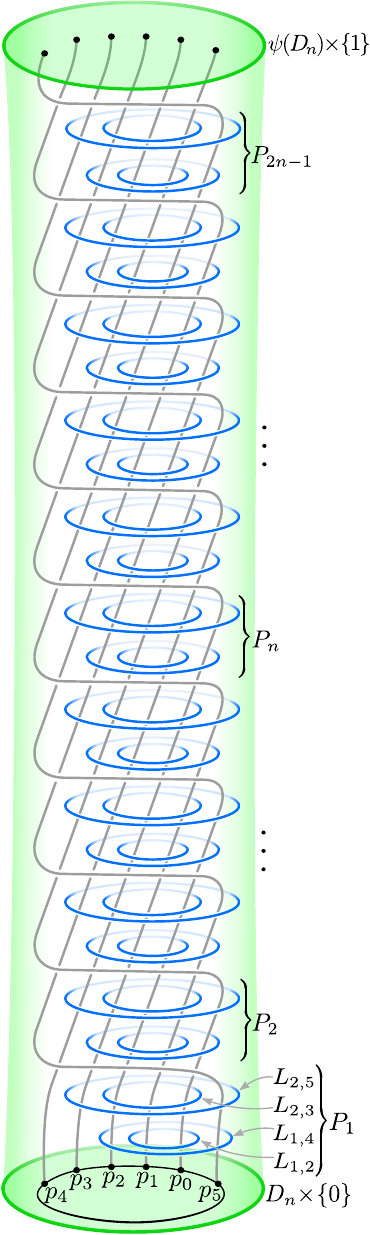}
   	\vspace{-5pt}
   	\caption{$T_\phi$, with loops $L_{i,j}$.}
  \label{fig:M_C}
\end{wrapfigure}

For $1\le i \le 4n$, and $2\le j \le n-2$, $j\not\equiv i \mod 2$, define $\beta_{i,j}$ to be the simple closed curves on $D_n$ that bounds a disk punctured by $p_k$, for $\frac{i-(j-1)}{2}\le k \le \frac{i+(j-1)}{2}$. Here we take $p_k$ indices modulo $n$. For $i$ even (so that $j$ is odd), $\beta_{i,j}$ is the curve that bounds $j$ punctures, centered at the puncture $p_{\rfrac{i}{2}}$. For $i$ odd ($j$ even), $\beta_{i,j}$  bounds $j$ punctures, centered between the two punctures $p_{\frac{i\pm 1}{2}}$.

For $1\le i \le 2n$, Let $P_i$ be the pants decomposition of $D_n$ defined by the collection of curves 
$$
\{\beta_{2i-1,2},\beta_{2i-1,4},\dots, \beta_{2i-1,n-2}, \beta_{2i,3},\beta_{2i,5},\dots, \beta_{2i,n-1}\}.
\qquad\qquad\qquad\qquad\qquad\qquad\qquad
$$

For the pants decomposition $P_i$ and its curves, indices are \emph{not} taken modulo $n$. This means that $\beta_{i,j}$ and $\beta_{i+n,j}$ describe the same curve on $D_n$, though we considered these to be distinct copies of this curve. Similarly, $P_i$ and $P_{i+n}$ define the same pants decomposition of $D_n$, but will be regarded as distinct copies of it.

\Cref{fig:P_1} shows $P_1$ for the case $n=6$; we get any other $P_i$ by rotating the curves of $P_1$ by an angle of $\frac{2\pi i}{n}$. Note that $P_i$ contains $n-2$ simple closed curves. For $0\le k\le n-2$, let $P_i^k$ be the pants decompositions consisting of the first $k$ curves of $P_{i+1}$ union the last $(n-2)-k$ curves of $P_i$. For $n=6$, the sequence of pants decompositions 
$$P_1=P_1^0, P_1^1, \dots, P_1^4=P_2 \qquad\qquad\qquad\qquad\qquad\qquad\qquad\qquad$$
 is shown in \Cref{fig:P1_to_P2}. In \Cref{fig:subpath} we show a simplified picture to save space---only the subdisk of $D_n$ cut out by the convex hull of the points $p_k$ is shown, and the $\beta_{i,j}$ curves are represented by transversely oriented arcs. The transverse orientation shows how to complete the arc to a curve---it should be an inward normal for the curve, and the rest of the curve should lie outside the hexagon.

The sequence of pants decompositions 
$$P_i=P_i^0,P_i^1,\dots,P_i^{n-2}=P_{i+1}\qquad\qquad\qquad\qquad\qquad\qquad\qquad\qquad$$
 is a path $\path_i$ in the pants graph from $P_i$ to $P_{i+1}$, of length $n-2$. The path $\path_i$ is obtained by shifting each of the curves of $P_i$, one at a time and in order, by a rotation of $2\pi/n$. Let $\path$ be the composition of the paths $\path_i$ for $1\le i\le 2n+1-l$. Note that the composition of paths $\path_1, \dots, \path_n$ is a loop of length $n(n-2)$ that begins at $P_1$ and ends at $P_{n+1}$, and the composition $\path_{n+1},\dots,\path_{2n}$ is a path of length $(n-l)(n-2)$ beginning at $P_{n+1}$ and ending at $P_{2n+1-l}$. Thus $\path$ is a path from $P_1$ to $P_{2n+1-l}$ of length $m:=(2n-l)(n-2)$. Since $P_1$ and $P_{n+1}$ are two copies of the same pants decomposition, we see that the path $\path$ results from traversing a loop in the pants graph, then partially traversing the same loop again, by an amount depending on $l$. 

\begin{figure}
	\begin{subfigure}{.25\textwidth}
 		\raggedright
   		\includegraphics[width=.95\textwidth]{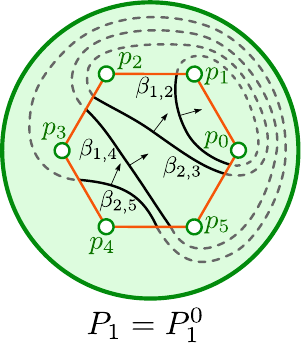}
   		\caption{}
   		\label{fig:P_1}
	\end{subfigure}
	\begin{subfigure}{.74\textwidth}
 		\raggedleft
   		\includegraphics[width=.95\textwidth]{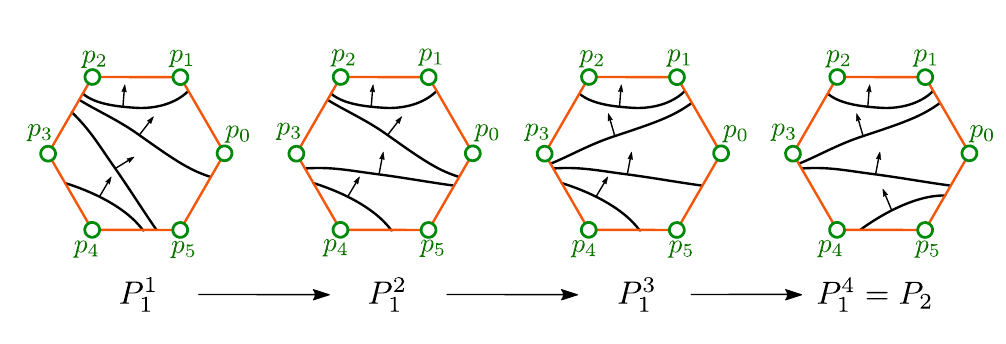}
   		\caption{}
   		\label{fig:subpath}
	\end{subfigure}	
	\caption{The sequence of pants decompositions from $P_1$ to $P_2$, for $n=6$. The drawing of $P_1$ on the left shows how to interpret the simplified drawings of $P_1^k$, $1 \le k\le 4$, on the right.}
	\label{fig:P1_to_P2}
\end{figure}

Since $\phi(P_1)=P_{2n+1-l}$, the construction of \Cref{sec:const} can be applied to the path $\path$, giving a compact manifold $M_C$. In particular, at each step of the path $\path$, a curve $\beta_{i,j}$ is replaced by another curve $\beta_{i',j'}$. If $\beta_{i,j}$ is replaced at step $k$, then let $L_{i,j} = \beta_{i,j}\times \{\frac{k}{m}\}$ for $1\le k\le m$. Then $M_C=T_\phi \setminus \bigcup_{i,j}\N(L_{i,j})$. We will denote by $B_{i,j}$ the boundary component of $M_C$ coming from removing $\N(L_{i,j})$ from $T_\phi$. Actually, it will be convenient for the exposition to come to modify the $L_{i,j}$ by an isotopy in $T_\phi$. In particular, we arrange so that $L_{i,j}$ lies in the fiber $D_n\times \{\frac{i}{2(2n+1-l)}\}$, and do the same for the $B_{i,j}$ boundary components of $M_C$. This isotopy only moves curves within a pants decomposition $P_k$ past each other, and such an isotopy exists since curves in $P_k$ are disjoint in $D_n$. The remaining boundary components of $M_C$, corresponding to the augmentation circle $L_q$ and the braid components $L_p$ of $L$, will be denoted by $B_q$ and $B_p$, respectively. Note that $B_p$ is potentially a union of multiple boundary components, but we will rarely need to refer to them individually. \Cref{fig:M_C} shows $T_\phi$ with the $L_{i,j}$ loops colored blue (alternatively, we can view these loops as representing boundary components $B_{i,j}$ in $M_C$). Note that a loop $L_{i,j}$ encircles $j$ strands of the braid $L_p$, and is at height $i$ relative to the other loops. We say that $B_{i,j}$ has \define{width} $j$. We will take the width of $B_q$ to be $n$, and the width of each component of $B_p$ to be 1.

\Cref{fig:pants} shows the pairs of pants along which $M_C$ decomposes into $A$-blocks, for part of $M_C$ (note that $M_C$ has no $S$-blocks). In this drawing, as in \Cref{fig:M_C}, we think of $B_q$ as the boundary of the solid torus containing $M_C$. The pants meeting $B_q$ are the outermost ones (in blue). Each of the pants is either a twice punctured disk or a punctured annulus, and the same pattern of pants continues in the rest of $M_C$. Here we are using the term \emph{puncture} somewhat non-standardly, to refer to a boundary component that meets a meridian of a component in $B_p$. 

\begin{figure}[h]
	\begin{subfigure}{.32\textwidth}
 		\centering
   		\includegraphics[width=.6\textwidth]{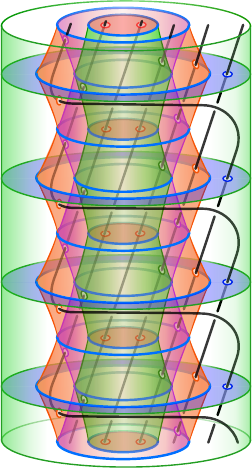}
   		\caption{}
   		\label{fig:pants}
	\end{subfigure}
	\begin{subfigure}{.32\textwidth}
 		\centering
   		\includegraphics[width=.67\textwidth]{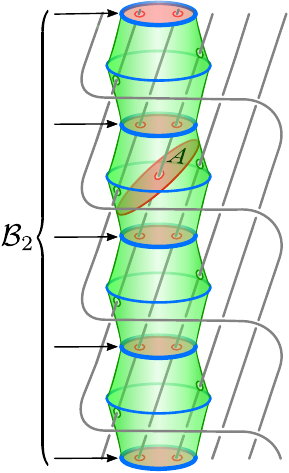}
   		\caption{}
   		\label{fig:C_2}
	\end{subfigure}	
	\begin{subfigure}{.32\textwidth}
 		\centering
   		\includegraphics[width=.6\textwidth]{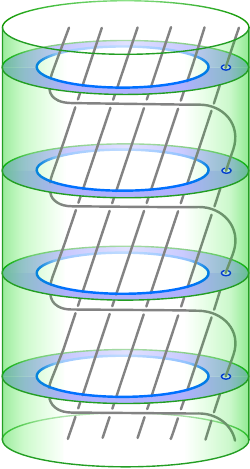}
   		\caption{}
   		\label{fig:C_n}
	\end{subfigure}	
	\caption{Left: $M_C$ decomposes into $A$-blocks along pants. With curves in $\B$ omitted, one can also interpret this as the complex $\C$. Center: the subcomplex $\C_2\subset \C$, with a compression annulus $A$. Right: the subcomplex $C_n\subset \C$. In the center and right picture, we show $B_p$ but omit filled tori that don't meet the subcomplex, to avoid cluttering.}
	\label{fig:subcomplexes}
\end{figure}

For each pair of pants $Y$ in $M_C$, let the \define{size} of $Y$ be the largest width of a horizontal boundary component that $Y$ meets. Thus in \Cref{fig:pants} the size 2 pants are the twice-punctured disks, for example. Note that size $k\ge 3$ pants always have another boundary component on a width $k-1$ loop.

%% file: proof.tex
\section{Small knots of large genus}\label{sec:proof}

In this section we prove \Cref{thm:main}. For the sake of clarity, we have relegated the bulk of this work to Lemmas \ref{lem:C_B}, \ref{lem:small}, and \ref{lem:heegaard}. In particular, in \Cref{lem:C_B} we get much of the technical work out of the way, and in Lemmas \ref{lem:small} and \ref{lem:heegaard} we then address smallness and Heegaard genus. At the end of the section we assemble the pieces to prove the theorem, which we restate below for convenience:

\begin{repthm}{thm:main}
For every $l\ge 1$, there exists a family $\{K_N\}_{N=4}^\infty$ of small $l$-component links such that $\SS^3\setminus K_N$ has Heegaard genus $n:=Nl$. Furthermore, $K_N$ is an $n$-bridge link, and can be constructed to have fewer than $4\pi n^5$ crossings.
\end{repthm}

The links we are after will result from long Dehn fillings along all $B_{i,j}$ boundary components of $M_C$, and along the component $B_q$. Let $\gamma=\{\gamma_q\} \cup \{\gamma_{i,j}\}_{i,j}$ be a system of filling slopes for these boundary components. Predictably, the slope for $B_{i,j}$ is $\gamma_{i,j}$, and the slope for $B_q$ is $\gamma_q$. We will require that every slope in the system $\gamma$ has the form $1/s$ for some $s\in \ZZ\setminus \{0\}$. We may choose $\gamma$ so that we avoid boundary slopes given by \Cref{thm:Hatcher}, and so that we avoid bad fillings as described in \Cref{sec:heegaard}. We may also take all slopes to be sufficiently long so that exceptional (i.e., non-hyperbolic) fillings are avoided. For the sake of convenience, we will refer to such a choice of $\gamma$ as a \define{generic} filling system. When we have fixed such a $\gamma$, we will denote by $M_\gamma$ the Dehn filling of $M_C$ along $\gamma$.

Let $\B_2\subset \del M_C$ consist of all $B_{i,j}$ of width 2, and fix some $\B\subset \B_2$. For such $\B$, define $L_\B=\{L_{i,j} \mid B_{i,j}\in \B\}$, and let $\Mf=M_\gamma\setminus \bigcup_{L_{i,j}\in L_\B} \N(L_{i,j})$. In other words, $\Mf$ is the manifold obtained by drilling out from $M_\gamma$ the surgery solid tori for $B_{i,j}\in \B$. Let $\B^c=\del M_C\setminus (\B \cup \{L_p\})$, viewed as a subset of $\Mf \supset M_C$. That is, $\B^c$ consists of all the $B_{i,j}$ that are filled in $\Mf$

Let $\C\subset \Mf$ be the 2-complex consisting of the pants in $M_C\subset \Mf$, and the tori $B_{i,j}\in \B^c\subset \Mf$ (see \Cref{fig:pants}). $\C$ carries all surfaces in $\Mf$ with horizontal boundary that are obtained by tubing pants. Next, define $\C_2$  to be the subcomplex of $\C$ consisting of all tori $B_{i,j}\in \B^c$ of width 2 and 3, and all pants of size 2 and 3. Define $\C_n\subset \C$ to be the subcomplex consisting of $B_q$, all tori $B_{i,j}$ of width $n-1$, and all pants of size $n$. The subcomplexes $\C_2$ and $\C_n$ are shown in \Cref{fig:C_2,fig:C_n}, respectively.

\begin{lem}\label{lem:C_B}
Fix any generic filling system $\gamma$ for $M_C$. Choose $\B \subset \B_2$, and define $\Mf$ as above. Let $(\Sigma, \del \Sigma)\subset (\Mf,\del \Mf)$ be an essential surface with horizontal boundary, and let $\piSig$ be the surface obtained by compressing $\Sigma$ along all compression annuli. Then 
\begin{itemize}
\item[(1)] $\Sigma$ is obtained by tubing together pants in $M_C$, 
\item[(2)] if $\B=\emptyset$ and no pants of $\Sigma$ is a twice-punctured disk, then $\piSig$ is carried by $\C_n$,
\item[(3)] if no pants of $\Sigma$ has boundary on $B_q$, then $\piSig$ is carried by $\C_2$.
\end{itemize}

\end{lem}

\begin{proof}

Let $\Sigma$ be an essential surface in $\Mf$. First we show that $\Sigma$ is isotopic into $M_C$. Let $B\in \B^c$, and let $\tau(B)$ be the filling solid torus for $B$. Consider a component $\alpha$ of the intersection of $\Sigma$ with $B$, which is necessarily a simple closed curve. First, if $\alpha$ is null-homotopic in $B$, then since $\Sigma$ is incompressible and $\Mf$ is irreducible, the disk $\alpha$ bounds in $B$ must bound a ball in $\Mf$ with another disk in $\Sigma$ (we may assume that $\alpha$ is innermost in $\Sigma$). In this case we can isotope $\Sigma$ to remove the intersection $\alpha$. If $\alpha$ is essential in $B$, then it cannot bound a disk, since $\gamma$ was chosen to avoid boundary slopes. In this case there must be another intersection curve $\alpha'$ isotopic to $\alpha$ in $B$, such that $\alpha$ and $\alpha'$ bound an annulus in $\Sigma\cap\tau(B)$. This annulus cuts off a solid torus with an annulus in $B$, and can be isotoped out of $\tau(B)$ across this solid torus, thus removing the intersections $\alpha$ and $\alpha'$ (here we assume the annulus in $\Sigma\cap \tau(B)$ is outermost in $\tau(B)$). It follows that $\Sigma$ is isotopic into $M_C$, and hence by \Cref{lem:tubing} it is obtained by tubing pants in $M_C$, thus establishing (1). Since $\piSig$ is obtained by compressing $\Sigma$ along compression annuli, by \Cref{lem:compress_annulus} it is incompressible. Let $Y_1,\dots, Y_k$ be the pants that tube together to give $\piSig$ (and therefore $\Sigma$).

A key observation is that when two pants $Y_1$ and $Y_2$ are tubed together along some $B\in \B^c$, the attached annulus can lie to either side of the filling torus, and can be isotoped across it. This is because the filling slope is of the form $\frac{1}{s}$, and therefore intersects each boundary component of the attached annulus once. More concretely, the solid torus is fibered as $D^2\times \SS^1$, and each fiber intersects the annulus in an arc, which can be isotoped across the disk. This shows how to build an isotopy of the annulus.

Each $B_{i,j}$ of width at least 3 and at most $n-2$ meets exactly $4$ pairs of pants in $\C$. Of these 4, let $Y_-$ and $Y_-'$ be the two of size $j$, and let $Y_+$ and $Y_+'$ be the two of size $j+1$. If $B_{i,j}\in \B^c$, let $A_{i,j}^-$ be the twice-punctured annulus resulting from tubing $Y_-$ and $Y_-'$ along $B_{i,j}$, and let $A_{i,j}^+$ be the result of tubing $Y_+$ and $Y_+'$ along $B_{i,j}$ (we will assume that the attaching annuli stay on $B_{i,j}$). Note that by the observation in the above paragraph, the choice of tubing of $Y_\pm$ and $Y_\pm'$ does not matter, since the attaching annulus can be isotoped across the surgery solid torus. If $B_{i,j}\in \B$, then we define  $A_{i,j}^-\coloneqq Y_-\cup Y_-'$ and $A_{i,j}^+\coloneqq Y_+\cup Y_+'$. For $B_{i,j}$ of width $n-1$, $B_{i,j}$ only meets 3 pairs of pants. In this case we can define $A_{i,j}^-$ in the same way as above, and identify $A_{i,j}^+$ with the pair of pants meeting $B_{i,j}$ and $B_q$. For $B_{i,j}$ of width 2, $B_{i,j}$ again meets only 3 pairs of pants, and we can define $A_{i,j}^+$ as above, and identify $A_{i,j}^-$ with the size 2 pants (i.e., the twice punctured disk) meeting $B_{i,j}$.

It follows that if $B_{i,j},B_{i,j+2}\in \B^c$ then $A_{i,j}^+$ is isotopic to $A_{i,j+2}^-$, by an isotopy across the $A$-block that they bound (see \Cref{fig:iso}). If, on the other hand, $B_{i,j}\in \B^c$ and $B_{i,j+2}\in \B$, then the same isotopy will take $A_{i,j}^+$ into $A_{i,j+2}^-\cup B_{i,j+2}$, which shows that $A_{i,j}^+$ has a compression annulus. Compressing this annulus results in a pair of once-punctured annuli isotopic to $A_{i,j+2}^-$. Similarly, if $B_{i,j+2}\in \B^c$ and $B_{i,j}\in \B$ then $A_{i,j+2}^-$ has a compression annulus, and after an annular compression is isotopic to $A_{i,j}^+$.

\begin{figure}[h]
 	\centering
   	\includegraphics[width=.6\textwidth]{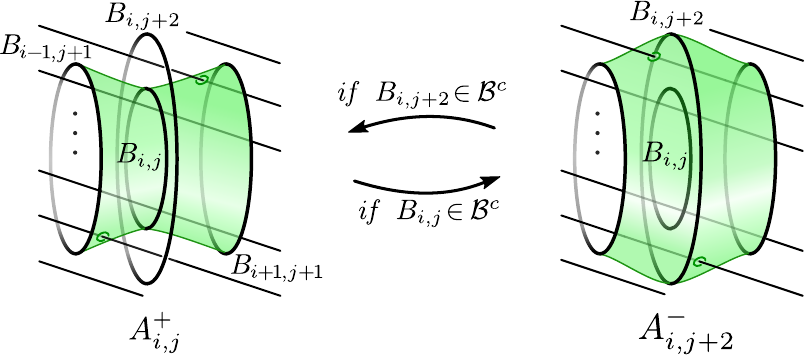}
   	\caption{If $B_{i,j}\in\B^c$, then $A_{i,j}^+$ can be isotoped onto $A_{i,j+2}^-$, possibly after an annular compression. Similarly, if $B_{i,j+2}\in \B^c$ then $A_{i,j+2}^-$ can be isotoped onto $A_{i,j}^+$, possibly after a compression.}
 \label{fig:iso}
\end{figure}

We are now ready to prove (2) and (3). First, assume that $\B=\emptyset$ and none of the pants $Y_1,\dots,Y_k$ are twice-punctured disks. Our goal will be to reduce $\C$ to the subcomplex $\C_n$, while ensuring that $\C_n$ also carries $\piSig$. First, since it is assumed that none of the $Y_i$ are twice-punctured disks, we may remove these from $\C$, and the resulting complex will still carry $\piSig$. With the twice-punctured disks removed, each $A_{i,2}^+$ for $1\le i \le 2n-l$ can be isotoped into $A_{i,4}^-$. The key here is that by removing the twice-punctured disks, we guarantee that the isotopy of $A_{i,2}^+$ extends to an isotopy of $\C$. In particular, since $B_{i,2}\in \B^c$, $B_{i,2}\setminus A_{i,2}^+$ is an annulus which we can isotope across the surgery solid torus and into $A_{i,2}^+$, thus allowing us to isotope $A_{i,2}^+$ across its $A$-block and into $A_{i,4}^-$. Thus we get a new complex that does not contain any $A_{i,2}^+$ or $B_{i,i+2}$. We continue in this way for all $A_{i,k}^+$, $k\le n-2$. Note that we must do this in order (with respect to $k$), since $A_{i,j}^+$ must be removed before $A_{i,j+1}^+$ can be removed. With the above reduction complete, only the subcomplex $\C_n$ remains, and $\piSig$ is still carried since we only changed $\C$ by removing unnecessary pants, and subcomplexes isotopic into $\C$, or isotopic into $\C$ after an annular compression.

(3) is proved similarly. If no pants of $\Sigma$ has boundary on $B_q$, then all pants meeting $B_q$ can be removed from $\C$, and $\C$ will still carry $\piSig$. With these pants removed, we can isotope punctured annuli $A_{i,n-1}^-$ down to $A_{i,n-3}^+$, and continue in this way until we are only left with tori of width $2$ and $3$, and annuli $A_{i,2}^+$, which is exactly the subcomplex $\C_2$. Note that since $\B\subset \B_2$, only tori of width 2 can be drilled out, so all of these isotopies are possible (or possibly compressions then isotopies in the case of $A_{i,4}^-$).
\end{proof}

\begin{lem}\label{lem:small}
For a generic filling system $\gamma$, $M_\gamma$ contains no closed incompressible surfaces.
\end{lem}

\begin{proof}
Suppose $\Sigma$ is a closed incompressible surface in $M_\gamma$. Applying \Cref{lem:C_B} with $\B=\emptyset$, it follows that $\Sigma$ is isotopic into $M_C$ and obtained by tubing pants. Let $\piSig$ be the surface obtained by compressing $\Sigma$ along all the compression annuli meeting $B_p$. This surface is obviously a tubing of pants isotopic into $M_C$, and by \Cref{lem:compress_annulus} it is also incompressible. 

Assume first that one of these pants is a twice-punctured disk. Then $\piSig$ must be a punctured sphere since all pants are either punctured disks or punctured annuli, and they are not tubed along punctures. Since $\piSig$ is isotopic into $M_C$, it is also isotopic into $\Mb$, which is a fibered manifold with fiber $D_n$ (in particular, it is an augmented braid closure). Although $\piSig$ may intersect every fiber of $\Mb$, there is some finite cyclic cover $\hatMb$ in which some fiber $F$ is disjoint from the lifts of $\piSig$. Since the only incompressible surface in the complement of a fiber is a fiber, it follows that the lifts of $\piSig$ are fibers and hence so was $\piSig$. But this is impossible since $\piSig$ only has boundary on $B_p$, and $B_q$ meets every fiber. Thus none of the pants is a twice-punctured disk.

\begin{figure}[h]
 	\centering
   	\includegraphics[width=.95\textwidth]{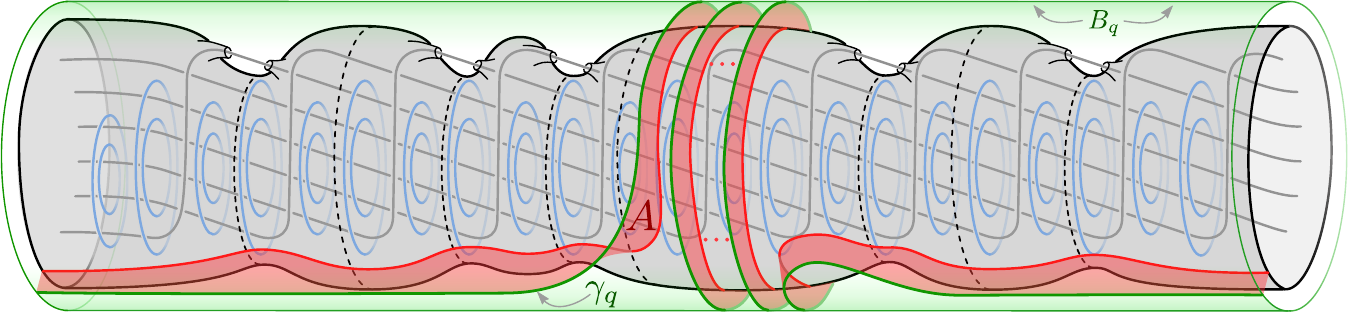}
   	\caption{A surface carried by the subcomplex $C_n$. Dotted curves show where to compress along annuli meeting $B_q$ and width $n-1$ loops to decompose into pants in $M_C$}
 \label{fig:compression}
\end{figure}

To finish the proof, we will return to viewing $\piSig$ in $M_\gamma$. Applying \Cref{lem:C_B} again with $\B=\emptyset$, $\piSig$ must be carried by the complex $\C_n$ since no pants is a twice-punctured disk. Any surface carried by $\C_n$ is obtained by tubing size $n$ pants along $B_q$ and $B_{i,n-1}$. A typical such surface is shown in \Cref{fig:compression}. This surface has an important property, which is shared by any surface obtained in this way. Namely, it has two annular subsurfaces, one that is isotopic to an annulus on $B_q$ parallel to its meridian, and one that is isotopic to an annulus on $B_q$ parallel to its longitude. It follows that there is an annulus $A$ having one boundary component on $\piSig$ and the other on $B_q$, such that the boundary on $B_q$ has slope $\gamma_q$ (the filling slope for $B_q$). Since $\gamma_q$ bounds a disk in the surgery solid torus filling $B_q$, the boundary of $A$ on $\piSig$ bounds a disk in $M_\gamma$, and clearly does not bound a disk on $\piSig$. Thus $\piSig$ is compressible, contradicting our assumption.
\end{proof}

\begin{lem}\label{lem:heegaard}
For a generic filling system $\gamma$ and $n\ge 4$, $M_\gamma$ has Heegaard genus $g_\gamma \ge n$.
\end{lem}

\begin{proof}
We will assume for a contradiction that $g_\gamma \le  n-1$. Since $\gamma$ is a generic filling system it contains no bad filling slopes. Consequently, each filling of a boundary component of $M_C$ by a solid torus is good and thus reduces the Heegaard genus by at most one (see \Cref{sec:heegaard}). Thus, if we drill out some of these solid tori in $M_\gamma$, then we increase the Heegaard genus by at most one for each drilling. 

In particular, we will drill out the filling solid tori for components $B_{i,2}$, starting with $B_{0,2}$, until the Heegaard genus increases as a result of drilling. This increase in Heegaard genus must happen for some filling, as we now show. Once again, we will let $\B \subset \B_2$ be the set of tori whose surgery solid torus has been drilled out, and define $\Mf$ as before. Since $\Mf$ has $l$ boundary component coming from $M_\gamma$, if we drill out all $2n-l$ such solid tori then it will have $2n$ boundary components, so its Heegaard genus will be at least $n=2n/2$. Thus the Heegaard genus of $\Mf$ must increase for at least one of these drillings. Let $B_{i_0,2}$ be the last torus drilled out, the drilling of which increases the Heegaard genus.

Now, let $H$ be a minimal genus Heegaard surface for the manifold $\Mf'$ resulting from re-filling $B_{i_0,2}$. By our choice of $\B$, $\mathrm{genus}(H)=g_\gamma \le n-1$. Since the Heegaard genus of $\Mf$ is strictly larger than that of $\Mf'$, $H$ is not a Heegaard surface for $\Mf$. Since $\gamma$ was chosen to avoid bad fillings, the core of the filling solid torus of $B_{i_0,2}$ is isotopic into $H$, and by cutting along it we get an almost Heegaard surface $H^\ast\subset \Mf$. The surface $H^\ast$ has two boundary components, and both are isotopic to longitudes of $B_{i_0,2}$. By \Cref{thm:RS}, there is an essential surface $\Sigma$ in $\Mf$ such that $\del\Sigma = \del H^\ast$, and $\mathrm{genus}(\Sigma)\le \mathrm{genus}(H^\ast)$. As noted following the statement of \Cref{thm:RS}, $\Sigma$ is obtained by compressing $H^\ast$, and may be disconnected. In the case that $\Sigma$ is disconnected, by $\mathrm{genus}(\Sigma)$ we mean the total genus of both components. Letting $\widehat{\Sigma}$ be the surface obtained by tubing together the boundary components of $\Sigma$ along $B_{{i_0},2}$, it follows that $\mathrm{genus}(\widehat{\Sigma})\le \mathrm{genus}(H)\le n-1$.

Since $\Sigma$ has horizontal boundary, \Cref{lem:C_B} implies that it is obtained by tubing pants in $M_C$. Let $\piSig$ and $\piSighat$ be the surfaces obtained by compressing $\Sigma$ and $\Sighat$, respectively, along all compression annuli having a boundary curve on $B_p$. Note that $\piSig$ and $\piSighat$ are also tubings of pants in $M_C$. If a component of $\piSighat$ contains a pants that is a twice-punctured disk, then by the argument in the proof of \Cref{lem:small} that component is a punctured sphere. All other components are punctured tori, since they are made of punctured annuli tubed together only at their boundaries (not their punctures). Let $\piSig'$ be the union of all connected components of $\piSig$ that contain some boundary component of $\piSig$. Note that by the preceding discussion, $\piSig'$ is either a union of two punctured disks (which tubes to a punctured sphere), or a punctured annulus (which tubes to a punctured torus). Note that $\piSig$ is incompressible by \Cref{lem:compress_annulus}, and hence so is $\piSig'$. Let $Y_1,\dots, Y_k$ be the pants in $M_C$ that tube together to give $\piSig'$.

\emph{Case 1: One of the pants $Y_i$ of $\piSig'$ has boundary on $B_q$.} Since $\piSig'$ has two boundary components on $B_{i_0,2}$, it must contain at least two pants meeting $B_{i_0,2}$. Therefore $\piSig'$ must contain enough pants to tube from $B_{i_0,2}$ up to $B_q$, then back down to $B_{i_0,2}$. This requires at least 2 pairs of pants of each size $j$, $3\le j \le n$, for a total of $2(n-2)$ pairs of pants. 

\begin{clm}
$\piSig'$ consists of exactly 2 pairs of pants of each size greater than or equal to 3.
\end{clm}

\begin{proof}[Proof of Claim]
If $\piSig'$ contains \emph{any} additional pants, then such pants will introduce at least two additional punctures, for a total of $2n-2$ punctures (note that there must be an even number of punctures, as they must tube in pairs). If one of the pants in $\Sigma_\circ'$ is size 2, then it must have two such, so we get an additional 2 punctures, and total is $2n$, so $\mathrm{genus}(\Sighat) \ge n$. If one of the pants is \emph{not} size 2, then $\Sigma_\circ'$ is a punctured annulus, whose boundary tubes together to give a punctured torus, so $\mathrm{genus}(\Sighat)\ge\frac{1}{2}(2n-2)+1=n$. Thus we conclude that $\Sigma$ has genus at least $n$, contradicting our assertion that $\mathrm{genus}(\Sighat)\le n-1$.
\end{proof}

It immediately follows that each pants of size $j$ for $4\le j \le n-1$ tubes to pants of size $j-1$ and $j+1$. The two size $n$ pants tube to each other, and to the two size $n-1$ pants, and the two size 3 pants meet the boundary $B_{i_0,2}$ and tube to the two size 4 pants. One possible configurations is shown in \Cref{fig:min_pants}. To rule out such a surface, we will need to demonstrate another isotopy of tubed pants. In particular, given a size $j$ pants and a size $j+1$ pants, with $3\le j\le n-1$, tubed together along their width $j$ boundary $B_{i,j}$, we have an isotopy which fixes the boundaries of these pants on $B_{i+ 1, j-1}$ and $B_{i+1, j+1}$, and takes them to two pants which are tubed along $B_{i+2,j}$. This is an isotopy of twice-punctured annuli, as shown in \Cref{fig:iso2}. Such an isotopy passes through the twice-punctured annulus that spans $B_{i+1,j-1}$ and $B_{i+1,j+1}$, and that is horizontal in the sense that it is contained in the level $i+1$ fiber if we view it in $T_\phi$. Note that we are able to isotope the tubed pants across $B_{i,j}$ and $B_{i+2,j}$ since these tori are filled and the slope of the pants boundaries are longitudes. 

\begin{figure}[h]
	\begin{subfigure}{.32\textwidth}
 		\raggedright
   		\includegraphics[width=.7\textwidth]{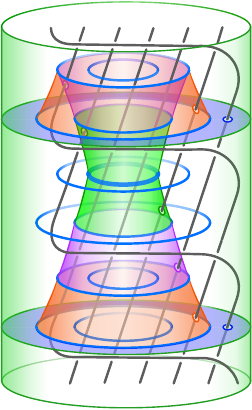}
   		\caption{}
   		\label{fig:min_pants}
	\end{subfigure}
	\begin{subfigure}{.32\textwidth}
 		\centering
   		\includegraphics[width=.77\textwidth]{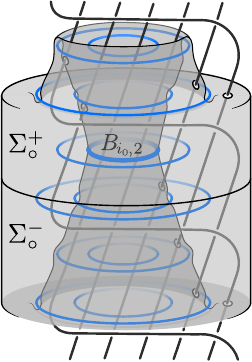}
   		\caption{}
   		\label{fig:piSig_pm}
	\end{subfigure}	
	\begin{subfigure}{.33\textwidth}
 		\raggedleft
   		\includegraphics[width=.8\textwidth]{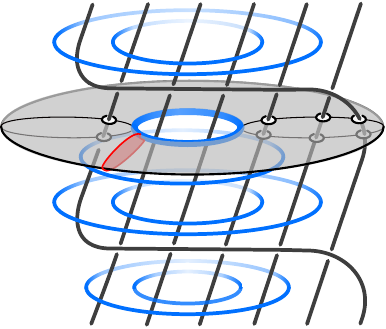}
   		\caption{}
   		\label{fig:collapsed} 
	\end{subfigure}	
	\caption{A subset of $\C$ consisting of exactly two pairs of pants of each size $\ge 3$ (left) tubes together to give a surface which can be cut into two pieces $\piSig^+$ and $\piSig^-$ (center), both isometric to horizontal annuli with $n-2$ punctures (right). By viewing $\piSig$ in this way we find a boundary compression (right).}
	\label{fig:subcomplexes}
\end{figure}

Now, let $\piSig^+$ be the half of $\piSig'$ containing exactly one pair of each size $j\ge 3$ of pants, and containing the size 3 pants having boundary on $B_{i_0+1,3}$. Let $\piSig^-$ be the other half (which also contains one pair of each size of pants). It follows that we can isotope $\piSig^+$ so that it is a tubing of the aforementioned horizontal twice-punctured annuli spanning $B_{i_0,j}$ and $B_{i_0,j+2}$, for $2\le j\le k$, where $k$ is $n-2$ is $n$ is even, and $k$ is $n-3$ if $n$ is odd. In the case where $n$ is odd, we will have in addition to these horizontal twice-punctured annuli the pants of size $n-1$. Altogether, these tube together to give a horizontal $(n-2)$ times punctured annulus. We can isotope $\piSig^-$ similarly, so that $\piSig'$ is as shown in \Cref{fig:collapsed}. 

We pause here to note that there are two possible ways of tubing together the pants meeting $B_q$, but both of the resulting surfaces are isotopic via an isotopy through the filling solid torus of $B_q$. So regardless of the choice of tubing, we can isotope $\piSig'$ so that it is as shown in \Cref{fig:collapsed}.

Returning to the proof, we find that $\piSig'$ must have a boundary compression disk, as shown in \Cref{fig:collapsed}. The demonstrated boundary compression is still a compression if we re-tube $\piSig$ along $B_p$ to obtain $\Sigma$, since it may be assumed to avoid a neighborhood of the punctures. Thus $\Sigma$ is not essential, contradicting our earlier assertion.

\begin{figure}[h]
 	\centering
   	\includegraphics[width=.82\textwidth]{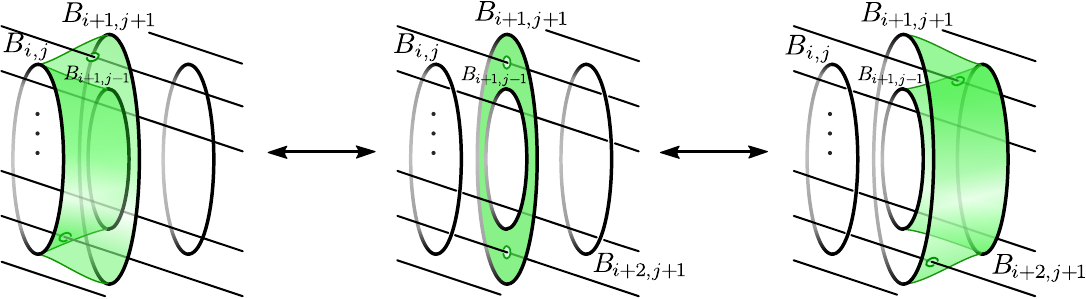}
   	\caption{When a pants of size $j$ and a pants of size $j+1$ share a boundary on $B_{i,j}$, and have their other boundaries on tori $B_{i+1,j\pm 1}$, there is an isotopy taking their $B_{i,j}$ boundary to $B_{i+2,j+1}$. Such an isotopy passes through a horizontal twice-punctured disk.}
 \label{fig:iso2}
\end{figure}

\emph{Case 2: No pants $Y_i$ has boundary on $B_q$.} By \Cref{lem:C_B} it follows that $\piSig$ is carried by $\C_2$. In this case $\piSig$ must contain as a subset a twice-punctured annulus $A_{i,i+3}^-$, for some $i$. But then $\piSig$ has a compression annulus $A$ which meets $B_p$, as shown in \Cref{fig:C_2}, contradicting our assumption that $\piSig$ was obtained by compressing it along \emph{all} compression annuli meeting $B_p$.	
\end{proof}

\begin{proof}[Proof of \Cref{thm:main}]

Fix $n\ge 4$. Let $\gamma=(\gamma_1,\dots, \gamma_m)$ be a system of Dehn filling slopes for the horizontal boundary components of $M_C$, with each slope of the form $\frac{1}{s_i}$. For any such system, the resulting manifold $M_\gamma$ is an $l$-component link exterior for the closure of an $n$-braid $K_n$. Since $n$-braid closures have Heegaard genus at most $n$, we deduce that $g(M_\gamma)\le n$. If we further assume that $k_i\ge 2\pi(2n-l)$ for each $i$, then by \cite[Theorem 1.1]{FP13} $(M_\gamma,H)$ is a good filling for any Heegaard surface of genus at most $n$ (in particular, for any minimal genus Heegaard surface). It follows from the discussion following Hatcher's theorem in \Cref{sec:essential} that if the system $\gamma$ is a boundary multi-slope, then for some $i$ increasing $s_i$ by 1 will give a slope that is not a boundary slope. Therefore we can find a filling system that is not a boundary multi-slope, and is a good filling, and such that $2\pi(2n-l)\le s_i \le 2\pi(2n-l)+2$. Such a choice of $\gamma$ is necessarily an exceptional filling by the $2\pi$-theorem of Gromov and Thurston \cite{BH96}.

Note that while \cite[Theorem 1.1]{FP13} is written in terms of the length $l(\frac{1}{s_i})$ of the filling slope in the induced Euclidean metric on a cusp cross-section, our statement follows since $l(\frac{1}{s_i})\ge s_i l(\lambda_i)\ge s_i$ for $\lambda_i$ the longitude of the $i^{th}$ boundary component.

For the choice of $\gamma$ described above, we have by \Cref{lem:small} that $M_\gamma$ contains no closed incompressible surfaces, and by \Cref{lem:heegaard} that $g(M_\gamma)\ge n$. It then follows that the Heegaard genus of $M_\gamma$ is exactly $n$. It also follows that $K_n$ must have bridge number $n$, for if it were smaller the Heegaard genus would have to be smaller as well.

In order to establish the claimed bound on the crossing number of $K_n$, we first observe that a $\frac{1}{s}$ Dehn filling along a boundary component of width $j$ is equivalent to adding $s$ full twists to the $j$ strands it encircles. One full twists of a band of $j$ strands produces $j(j-1)$ crossings, so $s$ twists produces $sj(j-1)$ crossings. Taking each $s_i$ in the range given in the previous paragraph, and taking into consideration that there are $2n-l$ loops of width $j$ for $2\le j \le n-1$, one loop of width $n$, and $(n-1)(2n-l)$ crossings coming from $L_p$, we get

\begin{align*}
c(K_n) &\le \overbrace{(n-1)(2n-l)}^{B_p}+\overbrace{(2\pi(2n-1)+2)n(n-1)}^{B_q} + \sum_{j=2}^{n-1} \overbrace{(2\pi(2n-1)+2)j(j-1)(2n-l)}^{B_{i,j}}\\
& < 4\pi n^2\left(n+2\sum_{j=2}^{n-1} j(j-1)\right) = 4\pi n^2 \left(n+ \frac{n(n-1)(2n-1)}{3}-n(n-1)\right) \le 4\pi n^5
\end{align*}
\end{proof}